\documentclass[10pt]{article}
\usepackage{amsmath, amssymb ,amsthm, amsfonts, amsgen}
\usepackage{graphicx}
\numberwithin{equation}{section} 
\def\O{{\Omega}}

\def\x{{\times}}

\def\G{{\Gamma}}

\def\E{{\exists}}

\def\H{{\cal H}}
\def\E{{\cal E}}

\newtheorem{Theorem}{Theorem}[section]
\newtheorem{Lemma}[Theorem]{Lemma}
\newtheorem{Proposition}[Theorem]{Proposition}

\newtheorem{Remark}[Theorem]{Remark}
\newtheorem{Definition}[Theorem]{Definition}
\newtheorem{Examples}[Theorem]{Examples}
\newcommand{{\rr}}{{\mathbb R}}

\newenvironment{@abssec}[1]{%
     \if@twocolumn
       \section*{#1}%
     \else
       \vspace{.05in}\footnotesize
       \parindent .2in
         {\upshape\bfseries #1. }\ignorespaces
     \fi}
     {\if@twocolumn\else\par\vspace{.1in}\fi}
\newcommand\keywordsname{Key words}
\newcommand\AMSname{AMS subject classifications}

\begin{document}

\author{Giuliano Gargiulo\footnote{Dipartimento di Scienze Biologiche ed Ambientali, Universit\'{a} degli Studi del Sannio, via Port'Arsa, 82100, Benevento, Italy.\newline Email:
\texttt{giuliano.gargiulo@unisannio.it}}\and Elvira Zappale\footnote{DIIMA,
Universit\`a degli Studi di Salerno, via Ponte Don Melillo, 84084
Fisciano (SA), Italy. \newline Email:
\texttt{ezappale@unisa.it}}}

\title{Some sufficient conditions for lower semicontinuity in SBD and applications to minimum problems of Fracture Mechanics.}
\date{}

\maketitle
\begin{abstract}
We provide some lower semicontinuity results in the space of special functions of bounded deformation for energies of the type  
$$
\int_{J_{u}} \Theta(u^+, u^-, \nu_{u})d \H^{N-1} \enspace, \enspace [u]\cdot \nu_u  \geq 0 \enspace {\cal H}^{N-1}-\hbox{ a. e. on }J_u,
$$
and give some examples and applications to minimum problems.

\end{abstract}
\medskip

\section{Introduction}

This study is motivated by the results contained in \cite{ ABF2, ABF3, AB} where it has been studied, both from the mechanical and computational viewpoint with several techniques, in the regime of linearized elasticity, the propagation of the fracture in a cracked body with a dissipative energy a la Barenblatt, i.e. of the type $\int_\Gamma \phi([u]\cdot \nu_u, [u] \cdot \tau_u)d {\cal H}^{N-1}$, where $\Gamma$ denotes the unknown crack site, $[u]\cdot \nu_u $, $[u] \cdot \tau_u$  represent the detachment and the sliding components respectively, of the opening of the fracture $[u]$ and, the energy density $\phi$ has the form
\begin{equation}\label{functionAB}
\phi([u]\cdot \nu_u,[u]\cdot \tau_u)=\left\{\begin{array}{ll}
0  &\hbox{if }[u] \cdot \nu_u=[u]\cdot \tau_u=0,\\
K &\hbox{if } [u] \cdot \nu_u \geq 0,\\
+\infty &\hbox{if } [u]\cdot \nu_u < 0 
\end{array}\right.
\end{equation}
where $K$ is a suitable positive constant.
\noindent
It has to be emphasized that the energy density $\phi$ in (\ref{functionAB}) also takes into account an infinitesimal noninterpenetration constraint, i.e. all the deformations $u$ pertaining to the effective description of the energy must satisfy $[u]\cdot \nu_u \geq 0 \; {\cal H}^{N-1}$ a.e. on $\Gamma$.

More precisely, the subsequent analysis aims to extend some of the results contained in \cite{GZ1, GZ2, GZ3}. In fact the target of those results was providing a mathematical justification to the minimization procedure adopted in \cite{ABF2, ABF3, AB} which appears at each time step, when studying the propagation of the fracture using the quasistatic evolution method, as introduced in \cite{FM} and developed in many other papers  (see for instance \cite{DMT1, DMT2} for the first formulation in terms of free discontinuity problem in the nonlinear elasticity setting, \cite{Ch} for the linear case, see also the more recent papers \cite{DMFT, FL, BFM} among a wide literature).
Indeed in order to derive, from the mathematical viewpoint, the properties of the energy $\phi$ above which guarantee lower semicontinuity with respect to the natural convergences (\ref{1.4BCDM}) $\div$ (\ref{1.6BCDM}) below, in order to generalize the energetic model contained  in \cite{ABF2, ABF3,AB} and finally to extend the lower semicontinuity results for surface integrals contained in \cite{BCDM}, the following result has been proved in \cite{GZ1}:
\begin{Theorem}\label{mainthmv}
Let $\O$ be a bounded open subset of $\mathbb R^N$, 
let 
\begin{equation}\label{famiglia}
\Phi:=\left\{\varphi :[0, + \infty[ \to [0, + \infty[, \varphi \hbox{ convex, subadditive and nondecreasing}\right\}
\end{equation} and let $\varphi \in \Phi$. 
Let $\{u_h\}$ be a sequence in $SBD(\O)$, such that $[u_h] \cdot \nu_{u_h} \geq 0$  $\H^{N-1}$-a.e. on $J_{u_h}$ for every $h$, converging to $u$ in $L^1(\O;\mathbb R^N)$ satisfying (\ref{1.3modrem2.3BCDM})  below, with a function  $\gamma :[0, +\infty[ \to [0, \infty[$ nondecreasing and verifying the superlinearity condition (\ref{1.2BCDM}) below.
Then 
\begin{equation}\label{interpenetration}
[u] \cdot \nu_u \geq 0 \enspace{ \cal H}^{N-1}- \hbox{ a.e. on }J_u,
\end{equation}
and
\begin{equation}\label{lscv}
\displaystyle{\int_{J_{u}} \varphi([u] \cdot \nu_{u})d \H^{N-1}\leq \liminf_{h \to + \infty} \int_{J_{u_h}} \varphi([u_h] \cdot \nu_{u_h})d \H^{N-1}.}
\end{equation}
\end{Theorem}

It can be easily seen that the class $\Phi$ in (\ref{famiglia}) includes functions of the type $\phi$ above, but  it has also to be remarked that, in general, the functions in $\Phi$ can be truly convex. Indeed, typical examples of functions in $\Phi$ are given by
$\varphi: s \in \mathbb R^+ \mapsto  (1 + s^p)^{\frac{1}{p}}$, $p \geq 1$, but in practice this class of functions does not perfectly fit the mechanical framework, where actually a `concave-type' behavior is expected. 

In fact, the present paper originates from the desire of finding a wider class of functions, containing the function $\phi$ in \cite{ ABF2, ABF3, AB}, including energy densities with a more general dependence on the opening of the fracture $[u]$ and on the normal of the crack site $\nu_u$ rather than just on their scalar product $[u]\cdot \nu_u$ as for $\varphi$ in (\ref{famiglia}) or possibly exhibiting a dependence from the `traces' on the two sides of the crack site, and which still ensures lower semicontinuity.
A first  result in this direction, i.e. the lower semicontinuity of $\int_{J_u}\Psi([u], \nu_u)d{\cal H}^{N-1}$ with respect to convergences (\ref{1.4BCDM}) $\div$ (\ref{1.6BCDM}), together with a characterization of such integrands (see \cite[Theorem 4.5]{GZ2}), has been achieved in \cite[Theorem 1.2]{GZ2}, where the class 
\begin{equation}\label{formulaPsi}
\Psi:(a,p) \in \mathbb R^N \times S^{N-1} \mapsto \sup_{\xi \in S^{N-1}}| p  \cdot \xi|\psi(|a\cdot \xi|),
\end{equation} 
has been introduced, with $\psi:[0, + \infty[ \to [0, + \infty[$ lower semicontinuous, nondecreasing,  subadditive (more generally a lower semicontinuos function such that $\psi(|\cdot|)$ is subadditive). For a more detailed discussion see Remark \ref{remarkclassinGZ2} below.

With the aim of considering surface energies whose densities have explicit dependence on the two different {\it one-sided Lebesgue limits} (see Section 2 below) and on the normal to the jump site, we introduce here the class $\Theta$ and prove Theorem \ref{mainthm} stated below. 
\begin{Definition}\label{Theta}
Let $\Theta$, with a notational abuse, be the class of functions of the form
\begin{equation}\label{Thetadef}
\displaystyle{\Theta:(i,j,p) \in \mathbb R^N \times \mathbb R^N \times S^{N-1} \to \sup_{\xi \in S^{N-1}}f(i\cdot \xi, j\cdot \xi, p\cdot \xi)}
\end{equation}
where $f:\mathbb R \to ]0,+\infty[$ is a continuous $BV$-elliptic function.
\end{Definition}

(See Definition \ref{BV-ellipticity} for $BV$-ellipticity.)

\begin{Theorem}\label{mainthm}
Let $\Omega$ be a bounded open subset of $\mathbb R^N$, let $\gamma:[0,+\infty[ \to [0,+\infty[$ be a non-decreasing function verifying condition  (\ref{1.2BCDM}) and let $\Theta$ be as in (\ref{Thetadef}) where $f$ is a continuous $BV$-elliptic function in the sense of Definition \ref{BV-ellipticity}. Let $\{u_h\}$ be a sequence in $SBD(\Omega)$ satisfying the bound  (\ref{1.3modrem2.3BCDM}), such that $[u_h]\cdot \nu_{u_h}\geq 0$, ${\cal H}^{N-1}$-a.e. on $J_{u_h}$ for every $h$ and converging to $u$ in $L^1(\Omega;\mathbb R^N)$. Then (\ref{interpenetration}) holds and
\begin{equation}\label{maineq}
\displaystyle{\int_{J_u}\Theta(u^+,u^-,\nu_u)d{\cal H}^{N-1}\leq \liminf_{h \to +\infty}\int_{J_{u_h}}\Theta(u_h^+,u^-_h, \nu_{u_h}) d{\cal H}^{N-1}.}
\end{equation}
\end{Theorem}

The structure of the paper is the following. In Section \ref{notazioni} the main results from Geometric Measure Theory concerning spaces of functions with bounded deformation and special functions of bounded variation, are recalled. Section 3 is devoted to Theorem \ref{mainthm} and to related minimun problems.

\section{Notations and preliminary results }\label{notazioni}


In this paper $\Omega$ will be a bounded open subset of $\mathbb R^N$.  We shall usually suppose, when not explicitly mentioned, (essentially to avoid trivial cases) that $N > 1$. Let $u \in L^1(\Omega;\mathbb R^m)$, the set of Lebesgue points of $u $ is denoted by $\Omega_u$. Equivalently $x \in \Omega_u$ if and only if there exists a (necessarily unique) $\tilde{u}(x)\in \mathbb R^m$ such that
$$
\lim_{\varrho \to 0^+}\frac{1}{\varrho^N}\int_{B_{\varrho}(x)}|u(y)-\tilde{u}(x)|d y=0.
$$

A function $u \in L^1(\Omega;\mathbb R^m)$ is said to be of bounded variation, and we write $u \in BV(\O;\mathbb R^m)$ if its distributional gradient $Du$ is an $m \times N$ matrix of finite Radon measures in $\O$, $Du \in {\cal M}_{b}(\O; M^{m\times N})$. Furthermore the following ${\it Lebesgue  (Radon-Nykodim)}$ decomposition holds
$Du= \nabla u {\cal L^N}+ D^s u$, where $D^s u$, the singular part with respect the Lebesgue measure ${\cal L}^N$, can be split as $ D^j u+ D^c u$, with $D^j u$ the restriction of $Du$ to $\Omega \setminus \O_u$, and $D^c u$ the restriction of $ D^s u$ to $\O_u$. For the above, and further details on BV functions, see e.g. \cite{AFP}.

$BD(\O)$ is the space of {\it vector fields with bounded deformation} and it is defined as the set of vector fields $u=(u^1,\dots, u^N) \in L^1(\Omega;\mathbb R^N)$ whose distributional gradient $Du=\{D_iu^j\}$ has the symmetric part
$$
Eu=\{E_{ij}u\}, E_{ij}u=(D_iu^j+D_ju^i)/2
$$
which belongs to ${\cal M}_b(\O;M^{N \x N}_{sym})$, the space of bounded Radon measures in $\O$ with values in $M^{N \x N}_{sym}$ (the space of symmetric $N \x N$ matrices). For $u \in BD(\O)$, the {\it jump set } $J_u$ is defined as the set of points $x \in \O$ where $u$ has two different {\it one-sided Lebesgue limits} $u^+(x)$ and $u^-(x)$, with respect to a suitable direction $\nu_u(x) \in S^{N-1}=\{\xi\in {\mathbb R}^N: |\xi|=1 \}$, i.e.
\begin{equation}\label{u+-}
\displaystyle{\lim_{\varrho \to 0^+}\frac{1}{\varrho^N}\int_{B^{\pm}_\varrho(x,\nu_u(x))}|u(y)-u^{\pm}(x)|dy=0},
\end{equation}
where $B^\pm_\varrho(x,\nu_u(x))=\{y \in \mathbb R^N:|y-x|< \varrho, (y-x)\cdot (\pm \nu_u(x))>0\}$, ($(u^+, u^-, \nu_u)$ are determined within permutation to $(u^-,  u^+, -\nu_u)$), accordingly we shall assume that all the subsequent integrands $f(i,j,p)$ will be compatible with this permutation, i.e. $f(i,j,p)= f(j,i,-p)$. Ambrosio, Coscia and Dal Maso \cite{ACDM} proved that for every $u \in BD(\O)$ the jump set $J_u$ is Borel measurable and countably $({\cal H}^{N-1}, N-1)$ rectifiable and $\nu_u(x)$ is normal to the approximate tangent space to $J_u$ at $x$ for $\H^{N-1}$-a.e. $x \in J_u$, where $\H^{N-1}$ is the $(N-1)$-dimensional Hausdorff measure (see \cite{AFP} and \cite{F}).

\noindent For every $u \in BD(\O)$, the Lebesgue decomposition of $Eu$ is  
$$
Eu=E^au+E^su
$$
with $E^au$ the absolutely continuous part  and $E^su$ the singular part with respect to the Lebesgue measure ${\cal L}^N$. $\E u$ denotes the density of $E^au$ with respect to ${\cal L}^N$, i.e. $E^au=\E u {\cal L}^N$. We recall that $E^s u$ can be  further decomposed as 
$$
E^s u= E^j u+ E^cu
$$
with $E^j u$, the {\it jump part} of $E u$, i.e. the restriction of $E^s u$ to $J_u$ and $E^c u$ the {\it Cantor part} of $E u$, i.e. the restriction of $E^s u$ to $\O \setminus J_u$. In \cite{ACDM} it has been shown that 
\begin{equation}\label{Ju}
E^j u= (u^+- u^-)\odot\nu_u \H^{N-1}\left \lfloor J_u \right.
\end{equation}
where $\odot$ denotes the symmetric tensor product, defined by $a \odot b:= (a \otimes b + b \otimes a)/2$ for every $a,b \in \mathbb R^N$, and $\H^{N-1}\left\lfloor J_u \right.$ denotes the restriction of $\H^{N-1}$ to $J_u$, i.e. $(\H^{N-1}\left\lfloor J_u \right.)(B)= \H^{N-1}(B \cap J_u)$ for every Borel set $B \subseteq \O$, (and we then write $B \in {\cal B}(\O)$). 
Moreover it has been also proved that $|E^c u|(B)= 0$ for every $B \in {\cal B}( \O)$ such that $\H^{N-1}(B)< + \infty$, where $|\cdot |$ stands for the total variation. In the sequel, for every $u \in L^1_{loc}(\O;\mathbb R^N)$ we denote by $[u]$ the vector $u^+ -u^-$.
\bigskip
For any $y, \xi \in \mathbb R^N$, $\xi \not = 0$, and any $B \in {\cal B}(\Omega)$ let
\begin{equation}\label{sections}
\begin{array}{l}
\pi_\xi := \{y \in \mathbb R^N: y \cdot \xi = 0\},\\
B_y^\xi:=\{t \in \mathbb R: y + t \xi \in B\}, \\
B^\xi:= \{y \in \pi_\xi: B^\xi_y \not = \emptyset\},
\end{array}
\end{equation} 
i.e. $\pi_\xi$ is the hyperplane orthogonal to $\xi$ , passing through the origin and $B^\xi= p_\xi(B)$, where $p_\xi$, denotes the orthogonal projection onto $\pi_\xi$. $B^\xi_y$  is the one-dimensional section of $B$ on the straight line passing through $y$ in the direction of $\xi$. 

Given a function $u : B \to \mathbb R^N$, defined on a subset $B$ of $\mathbb R^N$, for every $y, \xi \in \mathbb R^N$, $\xi \not =0$, the function $u^\xi_y:B^\xi_y \to \mathbb R$ is defined by
\begin{equation}\label{usezione}
u^\xi_y(t):= u^\xi(y + t \xi)=u(y + t \xi) \cdot \xi  \hbox{ for all } t \in B^\xi_y.
\end{equation}

\noindent Following \cite{ACDM} we can say that a vector field $u$ belongs to $BD(\O)$ if and only if its '{\it projected sections}' $u^\xi_y$ belong to $BV(\O^\xi_y)$. More precisely the following  Structure Theorem (cf. \cite[Theorem 4.5]{ACDM}) holds.
\begin{Theorem}\label{Structure}
Let $u \in BD(\Omega)$ and let $\xi \in \mathbb R^N$ with $\xi \not  = 0$. Then
\begin{itemize}
\item[(i)] $E^a u \xi \cdot \xi= \int_{\Omega^\xi}D^a u^\xi_y d {\cal H}^{N-1}(y), |E^a u \xi\cdot \xi|= \int_{\Omega^\xi} |D^a u^\xi_y|d {\cal H}^{N-1}(y)$.
\item[(ii)] For ${\cal H}^{N-1}$-almost every $y \in \Omega^\xi$, the functions $u^\xi_y$ and  $\tilde{u}^\xi_y$ (the Lebesgue representative of $u$, cf. formula (2.5) in \cite{ACDM}) belong to $BV(\Omega^\xi_y)$ and coincide ${\cal L}^1$-almost everywhere on $\Omega^\xi_y$, the measures $|D u^\xi_y|$ and $V \tilde{u}^\xi_y$ (the pointwise variation of $\tilde{u}^\xi_t$ cf. formula (2.8) in \cite{ACDM}) coincide on $\Omega^\xi_y$ ,and ${\cal E}u(y + t \xi)\xi \cdot\xi = \nabla u^\xi_y(t)=(\tilde{u}^\xi_y)'(t)$ for ${\cal L}^1$-almost every $t \in \Omega^\xi_y$.
\item[(iii)] $E^j u \xi \cdot\xi=\int_{\Omega^\xi}D^j u^\xi_y d {\cal H}^{N-1}(y)$, $|E^j u \xi\cdot \xi|=\int_{\Omega^\xi}|D^j u^\xi_y|d {\cal H}^{N-1}(y)$.
\item[(iv)]$(J_u^\xi)^\xi_y=J_{u^\xi_y}$ for ${\cal H}^{N-1}$-almost every $y \in \Omega^\xi$ and for every $t \in (J^\xi_u)^\xi_y$
$$
\begin{array}{ll}
u^+(y+ t \xi)\cdot \xi=(u^\xi_y)^+(t)=\lim_{s \to t^+}\tilde{u}^\xi_y(s)\\
u^-(y + t\xi)\cdot \xi=(u^\xi_y)^-(t)=\lim_{s \to t^-}\tilde{u}^\xi_y(s),
\end{array}
$$
where the normals to $J_u$ and $J_{u^\xi_y}$ are oriented so that $\nu_u \cdot \xi \geq 0$ and $\nu_{u^\xi_y}=1$.
\item[(v)]$E^c u \xi\cdot \xi=\int_{\Omega^\xi}D^c u^\xi_y d {\cal H}^{N-1}(y), |E^c u \xi \cdot \xi|=\int_{\Omega^\xi}|D^c u^{\xi}_y|d {\cal H}^{N-1}(y)$.
\end{itemize}
\end{Theorem}

The space $SBD(\O)$ of {\it special vector fields with bounded deformation} is defined as the set of all $u \in BD(\O)$ such that $E^cu=0$, or, in other words
$$
E u= \E u {\cal L}^N + [u]\odot \nu_u \H^{N-1}\left\lfloor J_u \right. 
$$

We also recall that if $\Omega \subset \mathbb R$, then the space $SBD(\O)$ coincides with the space of real valued special functions of bounded variations $SBV(\Omega)$, consisting of the functions whose distributional gradient is a Radon measure with no Cantor part (see \cite{AFP} for a comprehensive treatment of the subject).

Furthermore we restate \cite[Proposition 4.7]{ACDM} to be exploited in the sequel.
\begin{Proposition}\label{prop4.7}
Let $u \in BD(\Omega)$ and let $\xi_1, \dots, \xi_N$ be a basis of $\mathbb R^N$. Then the following three conditions are equivalent:
\begin{itemize}
\item[(i)] $u \in SBD(\Omega)$.
\item[(ii)] For every $\xi=\xi_i+ \xi_j$ with $1 \leq i, j \leq n$, we have $u ^\xi_y \in SBV(\Omega^\xi_y)$ for ${\cal H}^{N-1}$-almost every $y \in \Omega^\xi$.
\item[(iii)] The measure $|E^s u|$ is concentrated on a Borel set $B \subset \Omega$ which is $\sigma$-finite with respect to ${\cal H}^{N-1}$.
\end{itemize}
\end{Proposition}

Moreover, following \cite{ACDM} we give:
\begin{Definition}\label{Def4.9}
For any $u \in BD(\Omega)$ we define the non-negative Borel measure $\lambda_u$ on $\Omega$ as
\begin{equation}\label{lambdau}
\lambda_u(B):=\frac{1}{2 \omega_{N-1}}\int_{S^{N-1}}\lambda^\xi_u(B)d {\cal H}^{N-1}(\xi)  \;\; \forall B \in {\cal B}(\Omega),
\end{equation}
where, for every $\xi \in S^{N-1}$
\begin{equation}\label{lambdaxiu}
\lambda_u^\xi(B):= \int_{\Omega^\xi}{\cal H}^0(J_{u^\xi_y}\cap B^\xi_y)d {\cal H}^{N-1}(y) \;\; \forall B \in {\cal B}(\Omega).
\end{equation}
\end{Definition}

Let
\begin{equation}\label{1.6bis}
J_u^\xi:=\left\{x \in J_u: [u]\cdot \xi \not = 0\right\},
\end{equation}
we recall that

\begin{equation}\label{Junegligible}
{\cal H}^{N-1}(J_u \setminus J_u^\xi)=0 \hbox{ for }{\cal H}^{N-1}-\hbox{a.e. }\xi \in S^{N-1}. 
\end{equation}

The following result is a consequence of the Structure Theorem

\begin{Theorem}\label{thm4.10}
For every $u \in BD(\Omega)$ and any $\xi \in S^{N-1}$,
\begin{equation}\label{lambdaxiuarea}
\lambda^\xi_u(B)=\int_{J^\xi_u \cap B}|\nu_u  \cdot \xi|d {\cal H}^{N-1} \;\; \forall B  \in {\cal B}(\Omega),
\end{equation} 
where $\nu_u$ is the approximate unit normal to $J_u$. Moreover $\lambda_u= {\cal H}^{N-1}\lfloor J_u$. 
\end{Theorem}

\noindent 
A standard approximation argument by simple functions, proves, more generally, that for every Borel function $g:\Omega \to [0,+\infty]$, it results
\begin{equation}\label{changeofvariables}
\int_{J_u^\xi\cap B}g(y)|\nu_u \cdot \xi|d {\cal H}^{N-1}(y)=\int_{\Omega^\xi}\int_{p_\xi(J_u^\xi\cap B)}g(y+ t \xi)d{\cal H}^{0}(t)d {\cal H}^{N-1}(y)
\end{equation}
for any $\xi \in S^{N-1}$.

\medskip

We recall the following compactness result for sequences in $SBD$  proved in \cite[Theorem 1.1 and Remark 2.3]{BCDM}.

\begin{Theorem}\label{compactnessBCDM}
Let $\gamma:[0, + \infty[ \to [0, +\infty[$ be a non-decreasing function such that
\begin{equation}\label{1.2BCDM}
\displaystyle{\lim_{t \to +\infty}\frac{\gamma(t)}{t}=+ \infty.}
\end{equation} 
Let $\{u_h\}$ be a sequence in $SBD(\O)$ such that
\begin{equation}\label{1.3modrem2.3BCDM}
\|u_h\|_{L^{\infty}(\O;\mathbb R^N)}+ \int_{\O}\gamma(|\E u_h|)dx + \H^{N-1}(J_{u_h}) \leq K
\end{equation}
for some constant $K$ independent of $h$. Then there exists a subsequence, still denoted by $\{u_h\}$, and a function $u \in SBD(\O)$ such that
\begin{equation}\label{1.4BCDM}
u_h \to u \hbox{ strongly in }L^1_{loc}(\O;\mathbb R^N), 
\end{equation}
\begin{equation}\label{1.5BCDM}
\E u_h \rightharpoonup \E u \hbox{ weakly in }L^1(\O; M_{sym}^{N \x N}),
\end{equation}
\begin{equation}\label{1.6BCDM}
E^j u_h \rightharpoonup E^j u \hbox{ weakly* in } {\cal M}_b(\O; M_{sym}^{N \x N}),
\end{equation}
\begin{equation}\label{1.7BCDM}
\H^{N-1}(J_u) \leq \liminf_{h \to + \infty} \H^{N-1}(J_{u_h}).
\end{equation}
\end{Theorem}

\medskip

We will also make use of the following result from Measure Theory \cite[Lemma 2.35]{AFP}  
\begin{Lemma}\label{lemma2.35AFP}
Let $\lambda$ be a positive $\sigma$-finite Borel measure in $\Omega$ and let $\varphi_i:\Omega \to [0, \infty]$, $i \in \mathbb N$, be Borel functions. Then
$$
\int_\Omega \sup_i \varphi_i d \lambda = \sup  \left\{\sum_{i \in I} \int_{A_i}\varphi _i d \lambda\right\}
$$
where the supremum ranges over all finite sets $I \subset \mathbb N$ and all families $\{A_i\}_{i \in I}$ of pairwise disjoint open sets with compact closure in $\Omega$.
\end{Lemma}

Following \cite{A2}  (see also \cite[Definitions 5.13 and 5.17 respectively]{AFP}) we recall the notions of $BV$-ellipticity and joint convexity, (the first notion was already introduced in \cite{AB1, AB2} in order to describe sufficient conditions for lower semicontinuity in $SBV$ for surface integrals).

We stress that the definitions below we are referring to (i.e. $BV$-ellipticity and joint convexity)),  appear slightly different from those stated in \cite{AFP}, but we emphasize that for the applications to lower semicontinuity problems with respect to convergence (\ref{1.4BCDM}) $\div$ (\ref{1.7BCDM}) we have in mind, they can be considered as `equivalent'.
Indeed, what really matters to that aim, is to have the sequences $\{u_h\}\subset SBD(\O)$ with range in a suitable compact set of $\mathbb R^N$, (related to the considered energy density). This fact is evident in the arguments used in the proofs of lower semicontinuity results in the original articles (see \cite{A2} and also \cite{AFP}).

Let $Q_\nu$ be an open cube of $\mathbb R^N$, centred at $0$, with side lenght $1$ and faces either parallel or orthogonal to $\nu\in S^{N-1}$ and let $u_{i,j,\nu}$ be the function defined as
$u_{i,j,\nu}=\left\{\begin{array}{ll}
i &\hbox{ if }y \cdot \nu>0,\\
j &\hbox{ if } y \cdot \nu <0 \end{array}
\right.$.   
\begin{Definition}\label{BV-ellipticity}
Let $T\subset \mathbb R^m$ be a finite set, and $f:T\times T \times S^{N-1}\to[0,+\infty[$. A function $f$ is said to be $BV$-elliptic if
\begin{equation}\label{BVelliequ}
\int_{J_\nu}f(v^+,v^-,\nu_v)d {\cal H}^{N-1} \geq f(i,j,\nu)
\end{equation}
for any bounded piecewise constant function $v:Q_\nu \to T$ such that $\{v\not= u_{i,j,\nu}\}\subset \subset Q_\nu$ and any triplet $(i,j,\nu)$ in the domain of $f$.
\end{Definition}

A function $f: \mathbb R^m \times \mathbb R^m \times S^{N-1}\to[0, +\infty[$ is said $BV$-elliptic if it verifies (\ref{BVelliequ}) for any finite set $T \subset \mathbb R^m$.

In the sequel, with an abuse of notations we will use the same symbol for any $BV$-elliptic function and its positive $1$-homogeneous extension in the last variable.

\begin{Definition}\label{jointconvexity}
Let $f:\mathbb R^m \times \mathbb R^m \times \mathbb R^N\to [0, + \infty]$. We say that $f$ is jointly convex if
$$
f(i,j,p)=sup_{h \in \mathbb N}\{(g_h(i)-g_h (j))\cdot p\} \;\; \forall (i,j,p) \in \mathbb R^m\times \mathbb R^m \times \mathbb R^N,
$$
for some sequence $\{g_h\} \subset [C_0(\mathbb R^m)]^N$.
\end{Definition}

\noindent The above notion was introduced in \cite{A2} with the name of regular `bi-convexity', see Lemma 3.4 therein.

We also recall, as proven in \cite{A2} (see also \cite{AFP}), that joint convexity implies $BV$-ellipticity, and the equivalence between the two notions is still an open problem, even if there are some classes of function for which the two notions are proven to be equivalent (see \cite[Example 5.1]{A2} and Example 3.5 herein). On the other hand $BV$-ellipticity is very difficult to verify in practice, whereas this is not the case for joint convexity. Moreover, necessarily any jointly convex function is lower semicontinuous and $f(i,j,p)=f(j,i,-p), f(i,i,p)= 0 \;\; \forall i, j \in \mathbb R^m, p \in \mathbb R^N,$ $ f(i, j, \cdot)$  is positively $1$-homogeneous and convex, $\forall i,j \in \mathbb R^m$.

In \cite{A2} (see Theorem 3.3 therein) it has been proven the following theorem that will be invoked in the proof of Theorem \ref{mainthm}.

\begin{Theorem}\label{lscbvelliptA2}
Let $f:\mathbb R^m \times \mathbb R^m \times S^{N-1}\to [0, +\infty[$ be a continuous $BV$- elliptic function. Let $\{u_h\} \subset SBV(\Omega; \mathbb R^m)$ be a sequence converging in $L^1(\Omega;\mathbb R^m)$ to $u$ such that $\|u_h\|_{L^\infty}$  and ${\cal H}^{N-1}(J_{u_h})$ are bounded and $\{|\nabla u_h|\}$ is equiintegrable. Then $u \in SBV(\Omega;\mathbb R^d)$ and 
$$
\displaystyle{\int_{J_u} f(u^+, u^-, \nu_u)d {\cal H}^{N-1} \leq \liminf_h \int_{J_{u_h}} f(u^+_h, u^-_h, \nu_h) d {\cal H}^{N-1}.}
$$
\end{Theorem}

Assuming $f$ jointly convex, one can allow $f$ to take the value $+\infty$ and not necessarily be continous, as it has been proven in \cite[Theorem 3.6]{A2}, (see also \cite[Theorem 5.22]{AFP}.)

\begin{Theorem}\label{lscjcAFP}
Let $K' \subset \mathbb R^m$ be a compact set and let $f:\mathbb R^m \times \mathbb R^m \times S^{N-1}\to [0, +\infty]$ be a jointly convex function. Let $\{u_h\} \subset SBV(\Omega; \mathbb R^m)$ be a sequence converging in $L^1(\Omega;\mathbb R^m)$ to $u$ such that $u_h \in K'$ a.e. in $\O$, $\{|\nabla u_h|\}$ is equiintegrable and ${\cal H}^{N-1}(J_{u_h})$ bounded. Then $u \in SBV(\Omega;\mathbb R^d)$, $u \in K'$ a.e. in $\O$ and, 
$$
\displaystyle{\int_{J_u} f(u^+, u^-, \nu_u)d {\cal H}^{N-1} \leq \liminf_h \int_{J_{u_h}} f(u^+_h, u^-_h, \nu_h) d {\cal H}^{N-1}.}
$$
\end{Theorem}

We observe that the assumption $\inf f>0 $ is only needed in the proof of \cite[Theorem 3.3]{A2}, for (3.20) therein, which is actually a consequence of the hypotheses of the present Theorem \ref{lscbvelliptA2}. Similar considerations apply to our Theorem \ref{lscjcAFP}.
Actually, from the mechanical viewpoint, boundedness of the third term in (\ref{1.3modrem2.3BCDM}) above may be interpreted as a ban to fractures to fill the material.

Moreover we emphasize that, to our purposes, i.e. for Theorem \ref{lscjcAFP} we could replace $f$ defined on $\mathbb R^m\times \mathbb R^m \times S^{N-1}$ by a function defined just on the compact set $K' \times K' \times S^{N-1}$. On  the other hand it would be enough to require such a density jointly convex just on $K' \times K' \times S^{N-1}$ obtaining it through functions $\{g_h\} \subset (C(K'))^N$ in place of $\{g_h\} \subset (C_0(\mathbb R^m))^N$  when giving Definition \ref{jointconvexity}, since the sequence $\{u_h\}$ in Theorem \ref{lscjcAFP} has range in $K'$. In fact this latter approach has been followed in \cite{AFP}, but the present choice allows a more transparent comparison of the lower semicontinuity results Theorem \ref{mainthm} and Proposition \ref{risultatoconjoint} with the results contained in \cite{GZ1} and \cite{GZ2}, see Theorems \ref{mainthmv} herein and \cite[Theorem 1.2]{GZ2}.

\section{Theorem \ref{mainthm} and Applications}\label{secmainthm}

We start this section by providing a lower semicontinuity lemma  along directions that will be to a great degree exploited in the proof of Theorem  \ref{mainthm}. The proof develops in analogy with a similar result in \cite{GZ2}, essentially exploiting the slicing method for $SBD$ fields introduced in \cite{ACDM, BCDM}, and we write it here for reader's convenience.

\begin{Lemma}\label{lscalongdir}
Let $f$ be a continuous $BV$-elliptic function as in Definition \ref{BV-ellipticity}. Let $\Omega$ be a bounded open subset of $\mathbb R^N$. Let $\{u_h\}$ be a sequence in $SBD(\Omega)$ satisfying the bound (\ref{1.3modrem2.3BCDM}), such that $[u_h] \cdot \nu_{u_h} \geq 0$  ${\cal H}^{N-1}$-a.e. on $J_{u_h}$ for every $h$ and converging to $u$ in $L^1(\Omega;\mathbb R^N)$ .
Then
\begin{equation}\label{eqscilemma}
\begin{array}{ll}
\displaystyle{\int_{J_u} f\left(u^+(y)\cdot \xi, u^-(y)\cdot \xi,\nu_u \cdot \xi\right)d {\cal H}^{N-1}(y)\leq}\\
\displaystyle{ \liminf_h \int_{J_{u_h}} f\left(u_h^+(y)\cdot \xi, u_h^-(y)\cdot \xi,  \nu_{u_h}\cdot \xi \right)d {\cal H}^{N-1}(y)}
\end{array}
\end{equation}
for ${\cal H}^{N-1}$-a.e. $\xi \in S^{N-1}$.
\end{Lemma}
\begin{proof}

Let $\{u_h\}\subset SBD(\O)$ satisfying the bound (\ref{1.3modrem2.3BCDM}) and converging to $u$ in $L^1(\Omega; \mathbb R^N)$. Theorem \ref{compactnessBCDM} ensures that $u \in SBD(\O)$. 

Let $\xi \in S^{N-1}$, and let $p_\xi: J_u \to \pi_\xi$ be the orthogonal projection onto $\pi_\xi$.
First we observe that $(iv)$ in Theorem \ref{Structure} guarantees that one can choose the normals to $J_u$, $J_{u_h}$, $J_{u^\xi_y}$ and $J_{{u_h}^{\xi}_y}$ oriented so that $\nu_u \cdot \xi , \nu_{u_h}\cdot \xi\geq 0$ and $\nu_{u^\xi_y}= \nu_{{u_h}^{\xi_y}}=1$.

This fact and Proposition \ref{prop4.7} ensure that for ${\cal H}^{N-1}$- a.e. $y \in \Omega^\xi$   it results 
\begin{equation}\label{FLU}
\begin{array}{ll}
(u^\xi_y)^+(t)=(u\cdot \xi)^+(y + t \xi) \hbox{ and }(u^\xi_y)^-(t)=(u\cdot \xi)^-(y + t \xi) \hbox{ for every } t \in J_{u^\xi_y} \hbox{ and,} \\
({u_h}^\xi_y)^+(t)=(u_h\cdot \xi)^+(y + t \xi) \hbox{ and} ({u_h}^\xi_y)^-(t)=(u_h\cdot \xi)^-(y + t \xi) \hbox{ for every }t \in J_{{u_h}^\xi_y},
\end{array}
\end{equation} with $u^\xi_y, {u_h}_y^\xi \in SBV(\O^\xi_y)$ for ${\cal H}^{N-1}$-a.e. $y \in \O^\xi$.  

Thus we can restate (\ref{eqscilemma}) as
\begin{equation}\label{eqscilemma2}
\begin{array}{ll}
\displaystyle{\int_{J_u} |\nu_u \cdot \xi|f\left(u^+(y)\cdot \xi, u^-(y)\cdot \xi,1\right)d {\cal H}^{N-1}(y)\leq}\\
\displaystyle{ \liminf_h \int_{J_{u_h}} |\nu_{u_h}\cdot \xi|f\left(u_h^+(y)\cdot \xi, u_h^-(y)\cdot \xi,1 \right)d {\cal H}^{N-1}(y)}
\end{array}
\end{equation}

On the other hand, by  (\ref{1.6bis}) and (\ref{Junegligible}), 
we have
\begin{equation}\label{integralpsi(0)}
\begin{array}{ll}
\displaystyle{\int_{J_u}|\xi \cdot \nu_u| f(u^+(y)\cdot \xi, u^-(y \cdot \xi),1)d {\cal H}^{N-1}(y)= \int_{J^\xi_u}|\xi \cdot \nu_u| f(u^+(y)\cdot \xi, u^-(y)\cdot \xi,1)d {\cal H}^{N-1}(y),}\\
\displaystyle{\int_{J_{u_h}}|\xi \cdot \nu_{u_h}| f(u_h^+(y)\cdot \xi, u_h^-(y)\cdot \xi,1)d {\cal H}^{N-1}(y)=\int_{J^\xi_{u_h}}|\xi \cdot \nu_{u_h}| f(u^+_h(y) \cdot \xi, u_h^-(y)\cdot \xi,1)d {\cal H}^{N-1}(y)}
\end{array}
\end{equation} 
for every $h \in \mathbb N$ and for ${\cal H}^{N-1}$-a.e. $\xi \in S^{N-1}$.
(\ref{FLU}), (\ref{integralpsi(0)}), (\ref{changeofvariables}) guarantee the existence of $N \subset S^{N-1}$ such that ${\cal H}^{N-1}(N)=0$ and  
$$
\int_{J_u}|\xi \cdot \nu_u| f(u^+(y)\cdot \xi,u^-(y)\cdot \xi,1)d {\cal H}^{N-1}(y)=   \int_{\O^\xi}\Big[\int_{J_{u^\xi_y}}f((u^\xi_y)^+(t),(u^\xi_t)^- (t),1)d {\cal H}^0(t)\Big]d {\cal H}^{N-1}(y),
$$

$$
\int_{J_{u_h}}|\xi \cdot \nu_{u_h}| f(u_h^+(y)\cdot \xi, u_h^-(y )\cdot \xi,1)d {\cal H}^{N-1}(y)=   \int_{\O^\xi}\Big[\int_{J_{{u_h}^\xi_y}}f(({u_h}^\xi_y)^+(t),({u_h}^\xi_y)^-(t),1)d {\cal H}^0(t)\Big]d {\cal H}^{N-1}(y),
$$
for every $h \in \mathbb N$ and for every $\xi \in S^{N-1} \setminus N$.

Consequently the proof will be completed once we show that
\begin{equation}\label{equivalenteqscilemma}
\begin{array}{ll}
\displaystyle{\int_{\O^\xi}\Big[\int_{J_{u^\xi_y}}f((u^\xi_y)^+(t),(u^\xi_t)^-(t),1)d {\cal H}^0(t)\Big]d {\cal H}^{N-1}(y)\leq }\\  \displaystyle{\liminf_{h \to +\infty}\int_{\O^\xi}\Big[\int_{J_{{u_h}^\xi_y}}f(({u_h}^\xi_y)^+(t),({u_h}^\xi_y)^-(t),1)d {\cal H}^0(t)\Big]d {\cal H}^{N-1}(y)}
\end{array}
\end{equation}
for every $\xi \in S^{N-1}\setminus N$.

To this end, for each $\xi \in S^{N-1}\setminus N$ consider a subsequence $\{u_k\}\equiv\{u_{h_k}\}$ such that  
\begin{equation}\label{subseqliminf}
\liminf_{h \to + \infty}\int_{J_{{u_h}^\xi_y}}f(({u_h}^\xi_y)^+(t),({u_h}^\xi_y)^-(t),1)d{\cal H}^0(t)=\lim_{k \to +\infty}\int_{J_{{u_k}^\xi_y}}f(({u_k}^\xi_y)^+(t),({u_k}^\xi_y)^-(t),1)d{\cal H}^0(t).
\end{equation}
Next consider a further subsequence (denoted by $\{u_j\}\equiv \{u_{k_j}\}$) such that
\begin{equation}
\lim_{j \to + \infty}{\cal H}^{N-1}(J_{u_j})=\liminf_{k \to + \infty}{\cal H}^{N-1}(J_{u_{k}}).
\end{equation}

\noindent
We want to show that the assumptions of Theorem \ref{lscbvelliptA2} in dimension one are satisfied.

By (ii) in Theorem \ref{Structure} (i.e. ${\cal E}u_j(y +t \xi)\cdot \xi = ({u_j}^\xi)'_y(t)$ for ${\cal H}^{N-1}$-a.e. $y \in \O^\xi$ and for ${\cal L}^1$-a.e. $t \in \O^\xi_y$) and by Fubini-Tonelli's theorem, for any $\xi \in S^{N-1}\setminus N$ we can define 
$I_{y,\xi}(u_j)=\int_{\O^\xi_y}\gamma(|{u'_j}^\xi_y(t)|)dt$, where ${u_j}^\xi_y(t)= u_j(y+t \xi)\cdot \xi$
and we have
$$
\int_{\pi_\xi}I_{y,\xi}(u_j)d {\cal H}^{N-1}(y)= \int_\O \gamma(|{\cal E}u_j(x)\xi \cdot \xi|)dx.
$$
Since $\{u_j\}$ satisfies the bound (\ref{1.3modrem2.3BCDM}) and $\gamma$ is non-decreasing, it follows that
\begin{equation}\label{1hyplemma2.1}
\int_{\pi_\xi}I_{y,\xi}(u_j)d{\cal H}^{N-1}(y)\leq \int_\O \gamma(|{\cal E}u_j(x)|)dx \leq K,
\end{equation}
for every $\xi \in S^{N-1}\setminus N$ and for ${\cal H}^{N-1}$-a.e. $y \in \O^\xi$.
It is also easily seen that, from the bound on $\|u_j\|_{L^\infty}$, deriving from the global bound (\ref{1.3modrem2.3BCDM}),
\begin{equation}\label{2hyplemma2.1}
\|{u_j}^\xi_y\|_{L^\infty(\O^\xi_y)}\leq K.
\end{equation}
From (\ref{1hyplemma2.1}), (\ref{changeofvariables}) and (\ref{1.3modrem2.3BCDM}) for every  $\xi \in S^{N-1}\setminus N$ it  results that there exists a constant $C \equiv C(K)$ such that
\par
$$
\liminf_{j \to +\infty}\int_{\pi_\xi}[I_{y,\xi}(u_j) + {\cal H}^0(J_{{u_j}^\xi_y})]d {\cal H}^{N-1}(y) \leq C< +\infty.
$$                                                                                                       Let us fix $\xi \in S^{N-1} \setminus N$  (such that the previous inequality holds).
Using Fubini-Tonelli's theorem and convergence in measure for $L^1$- converging sequences, we can extract a subsequence $\{u_m\}= \{u_{j_m}\}$ (depending on $\xi$) such that
 \begin{equation}\label{2.9bcdm}
 \begin{array}{ll}
\displaystyle{ \lim_{m \to + \infty} \int_{\pi_\xi}[ I_{y,\xi}(u_m)+ {\cal H}^0(J_{{u_m}^\xi_y})]d {\cal H}^{N-1}(y)=}\\
 \displaystyle{\liminf_{j \to + \infty}\int_{\pi_\xi}[ I_{y,\xi}(u_j)+ {\cal H}^0(J_{{u_j}^\xi_y})]d {\cal H}^{N-1}(y) \leq C < + \infty,}
 \end{array}
 \end{equation}
 and for a.e. $y \in \O^\xi$, $u_{m, y}^ \xi\in SBV(\O^\xi_y)$ and ${u_m}^\xi_y \to u_y^\xi$ in $L^1_{loc}(\O^\xi_y)$, with $u^\xi_y \in SBV(\O^\xi_y)$.

Let $\xi \in S^{N-1}\setminus N$: by (\ref{2.9bcdm}) and Fatou's lemma, for ${\cal H}^{N-1}$-a.e. $y \in \O^\xi$, it results
\begin{equation}\label{2.11bcdm}
\displaystyle{\liminf_{m \to + \infty}[I_{y,\xi}(u_m)+ {\cal H}^0(J_{{u_m}^\xi_y}) ]< +\infty.}
\end{equation}

Let us fix $N_{\O^\xi} \subset \O^\xi$ and a point $ y \in \Omega^\xi \setminus N_{\Omega^\xi}$,  such that ${\cal H}^{N-1}(N_{\Omega^\xi})=0$, (\ref{2.11bcdm}), (\ref{2hyplemma2.1}) hold and such that ${u_m}^\xi_y \in SBV(\O^\xi_y)$ for any $m$. Passing to a further subsequence $\{u_l\}\equiv \{u_{m_l}\}$ we can assume that there exists a constant $C'$ such that
$$
\displaystyle{\liminf_{m \to + \infty}[I_{y,\xi}(u_m)+ {\cal H}^0(J_{{u_m}^\xi_y}) ]= \lim_{l \to + \infty}[I_{y,\xi}(u_l)+ {\cal H}^0(J_{{u_l}^\xi_y}) ] \leq C'}.
$$ 

This means that $\{{u_l}^\xi_y\} \in SBV(\O_y^\xi)$ and satisfies all the assumptions of Theorem \ref{lscbvelliptA2} for each interval (connected component) $I \subset \O^\xi_y$. 
Consequently (\ref{subseqliminf}), $(iv)$ of Theorem \ref{Structure} and, Theorem \ref{lscbvelliptA2}  guarantee that
\begin{equation}\label{1dlsc}
\begin{array}{ll}
\displaystyle{\int_{J_{u^\xi_y}}f((u^\xi_y)^+(t),(u^\xi_y)^-(t),1)d {\cal H}^0(t)\leq \lim_{l \to + \infty}\int_{J_{{u_l}^\xi_y}}f(({u_l}^\xi_y)^+(t),({u_l}^\xi_y)^-(t),1)d {\cal H}^0(t)=}\\
\\
\displaystyle{\liminf_{h \to +\infty}\int_{J_{{u_h}^\xi_y}}f(({u_h}^\xi_y)^+(t),({u_h}^\xi_y)^-(t),1)d {\cal H}^0(t)}
\end{array}
\end{equation} 
for ${\cal H}^{N-1}$-a.e. $\xi \in S^{N-1}$ and for ${\cal H}^{N-1}$-a.e. $y \in \O^\xi$.

The lower semicontinuity stated in (\ref{equivalenteqscilemma}) now follows from Fatou's lemma, which completes the proof.

\end{proof}

\noindent
Now we are in position to prove Theorem \ref{mainthm}.

\begin{proof}[Proof of Theorem \ref{mainthm}]
We preliminarly observe that (\ref{interpenetration}) follows by Theorem \ref{mainthmv}, thus it only remains to prove (\ref{maineq}) and this will be achieved essentially through the applications of Lemma \ref{lscalongdir} and Lemma \ref{lemma2.35AFP}.

The continuity of $f$ allows us to assume $\xi$ in (\ref{Theta}) varying in any countable subset of $S^{N-1}$. It will be chosen in $S^{N-1} \setminus N$, $N$ being the ${\cal H}^{N-1}$ exceptional set introduced in Lemma \ref{lscalongdir}, and it will be denoted by ${\cal A}$, with elements $\xi_\alpha$. 

By superadditivity of liminf: 

$$\displaystyle{\liminf_{h \to + \infty} \int_{J_{u_h}}\Theta (u_h^+, u_h^-, \nu_{u_h})d \H^{N-1} \geq \sum_\alpha \liminf_{h \to +\infty}  \int_{J_{u_h}\cap A_{\alpha}}f(\xi_\alpha \cdot u_h^+,\xi_\alpha \cdot u_h^-, \xi \cdot \nu_{u_h})d \H^{N-1}}$$ for any finite family of pairwise disjoint open sets $A_\alpha \subset \Omega$. 

By Lemma \ref{lscalongdir} we have
$$
\displaystyle{\liminf_{h \to + \infty} \int_{J_{u_h}}f(\xi_\alpha \cdot u_h^+,\xi_\alpha \cdot u_h^-, \nu_{u_h})d \H^{N-1} \geq \int_{J_u} f(\xi_\alpha \cdot u^+,\xi_\alpha \cdot u_h^-, \nu_{u_h} d \H^{N-1} }
$$
for every $\xi_\alpha \in {\cal A}$.
Therefore
$$
\displaystyle{\liminf_{h \to + \infty} \int_{J_{u_h}}\Theta (u_h^+, u_h^-, \nu_{u_h})d \H^{N-1} \geq \sum_{\alpha} \int_{J_u \cap A_\alpha} f(\xi_\alpha \cdot u^+, \xi_\alpha \cdot u^-, \xi_\alpha \cdot \nu_{u}) d \H^{N-1} }
$$
for every $\xi_\alpha \in {\cal A}$ and for any finite family of pairwise disjoint open sets $A_\alpha \subset \Omega$. 

By Lemma \ref{lemma2.35AFP} we can interchange integration and supremum over all such families, thus getting
$$
\displaystyle{\liminf_{h \to + \infty}\int_{J_{u_h}}\Theta(u_h^+, u_h^-, \nu_{u_h})d \H^{N-1} \geq \int_{J_u}\Theta(u^+,u^-, \nu_u )d \H^{N-1}},
$$
whence (\ref{maineq}) follows and this concludes the proof.

\end{proof}

\begin{Remark}\label{remarkclassinGZ2}

It is worthwhile to observe that $\phi$ in (\ref{functionAB}) of \cite{ABF2, ABF3, AB} can be recast in terms of a suitable $\Theta$ in (\ref{Thetadef}) requiring in the model that the noninterpenetration constraint (\ref{interpenetration}) is verified. In fact it suffices to consider (as already observed in \cite{GZ2}) 
$$
\displaystyle{\Theta(i,j,p)= \sup_{\xi \in S^{N-1}} K |p \cdot \xi|},
$$
i.e. $f(a_1,a_2, b)= \psi(|a_1-a_2|)\theta(b)$ for suitable $\psi$ and $\theta$ (see 2 of Examples \ref{EXAIMETA} below), with  $\psi= \psi_{\rm const}: t \in [0, +\infty[ \to K, K>0$, and $\theta= |\cdot|$, from which one deduces that $\Theta= \Theta_{\rm const}: (i,j,p)\in \mathbb R^N \times \mathbb R^N\times S^{N-1}\to K$.  

Moreover we recall as emphasized in \cite[Remark 4.8]{GZ2} that the constant functions $K$ represent the only intersections between the classes $\Psi$ in (\ref{formulaPsi}) and $\Phi$ in (\ref{famiglia}). 
On the other hand, the fact that the classes (\ref{famiglia}) and (\ref{Thetadef}) do differ is not very surprising and, indeed, also the techniques adopted to prove the related lower semicontinuity results Theorem \ref{mainthmv} and Theorem \ref{mainthm} (and its simplified version given in \cite[Theorem 1.2]{GZ2}) are very different, the first relying essentially on Geometric Measure Theory and the second on the structure of the Special fields with Bounded Deformation together with the characterization of lower semicontinuity in $SBV$, enlightened in \cite{ACDM, BCDM} and in \cite{A2}.

\end{Remark}

We observe that, while joint convexity entails $BV$-ellipticity, on the other hand, one can replace the $BV$-elliptic function $f$ in (\ref{Thetadef}) by a jointly convex one, which may take also the value $+\infty$, i.e.
\begin{equation}\label{Thetadef2} 
\displaystyle{\Theta:(i,j,p) \in \mathbb R^N \times \mathbb R^N \times S^{N-1} \to \sup_{\xi \in S^{N-1}}f(i\cdot \xi, j\cdot \xi, p\cdot \xi)}
\end{equation}
where $f:\mathbb R \times \mathbb R \times \mathbb R \to ]0,+\infty]$ is a jointly convex function 
as in Definition \ref{jointconvexity}.


Thus the following result holds, which is indipendently obtained and not stated as a Corollary of Theorem \ref{mainthm}, since we may  avoid to require $f$ continuous and finite.

\begin{Proposition}\label{risultatoconjoint}
Let $\Omega$ be a bounded open subset of $\mathbb R^N$, let $\gamma:[0,+\infty[ \to [0,+\infty[$ be a non-decreasing function verifying condition  (\ref{1.2BCDM}) and let $\Theta$ be as in (\ref{Thetadef2}) where $f$ is a jointly convex function.
Let $\{u_h\}$ be a sequence in $SBD(\Omega)$ satisfying the bound  (\ref{1.3modrem2.3BCDM}), such that $[u_h]\cdot \nu_{u_h}\geq 0$, ${\cal H}^{N-1}$-a.e. on $J_{u_h}$ for every $h$ and converging to $u$ in $L^1(\Omega;\mathbb R^N)$. Then (\ref{interpenetration}) and (\ref{maineq}) hold. 
\end{Proposition}
\begin{proof}[Proof]
First we assume $f$ continuous.  Under this extra assumption, the proof develops as in Theorem \ref{mainthm} making use of Theorem \ref{lscbvelliptA2} in place of Theorem \ref{lscjcAFP}, when stating and proving the analogue of Lemma \ref{lscalongdir}. 

Then for general  jointly convex $f$, 
it is enough to observe that by Definition \ref{jointconvexity}, $f$ can be approximated by a non decreasing sequence of continuous jointly convex functions, namely $f_k(a_1,a_2,b)= \sup_{h \leq k}\{(g_h(a_1)-g_h(a_2))\cdot b\}$.

Furthermore, for every $k \in \mathbb N$, let 	$\Theta_k : \mathbb R^N \times \mathbb R^N \times  S^{N-1} \to  [0;+\infty]$ be the functional defined by
	\begin{equation}\label{3.14GZ2}
	\Theta_k(i,j, p) := \sup_{\xi \in S^{N-1}} f_k(i \cdot \xi, j\cdot \xi, p \cdot \xi)
	\end{equation}
Clearly,
\begin{equation}\label{3.15GZ2}
	\Theta(i,j,p) =\sup_{k \in \mathbb N}
	\Theta_k(i,j, p).
	\end{equation}
Since this supremum is actually a monotone limit, monotone convergence theorem gives
$$
\displaystyle{\int_{J_u} \Theta(u^+,u^-, \nu_u)  d{\cal H}^{N-1} = \lim_{k \to +\infty} \int_{J_u}
	\Theta_k(u^+, u^-, \nu_u)d{\cal H}^{N-1}}
	$$
On the other hand, the first part of the proof ensures that each functional
$\int_{J_u} \Theta_k(u^+, u^-; \nu_u)d{\cal H}^{N-1} $ is sequentially lower semicontinuous with respect to the $L^1$- strong convergence along all the sequences $\{u_n\}\in SBD(\Omega)$ satisfying the bound (\ref{1.3modrem2.3BCDM}), so that
$$
\displaystyle{\int_{J_u}\Theta(u^+,u^-,\nu_u)d{\cal H}^{N-1} \leq \liminf_{k \to +\infty}
\int_{J_{u_h}}\Theta(u_h^+. u_h^-, \nu_{u_h}) d {\cal H}^{N-1}}   
$$
which concludes the proof.
\end{proof}

\begin{Remark}\label{osssuProposition}
We emphasize that Theorem \ref{mainthm} and Proposition \ref{risultatoconjoint} still hold with obvious adaptations if one replaces the integrand $\Theta$ in (\ref{Thetadef}) (or (\ref{Thetadef2}) respectively) by
$$
\Theta(i,j,p):= \sup_{\xi \in S^{N-1}} f_\xi(i\cdot \xi, j \cdot \xi, p \cdot \xi )
$$
with $f_\xi$ as in (\ref{Thetadef}) (or (\ref{Thetadef2}) respectively) continuously depending on $\xi \in S^{N-1}$.

It is worthwhile to observe that, looking at the proof of Lemma \ref{lscalongdir}, Proposition \ref{risultatoconjoint} provides lower semicontinuity along sequences $\{u_h\}$ satisfying (\ref{1.3modrem2.3BCDM}) also for energy densities $\Theta$ obtained via (\ref{Thetadef2}) by functions $f$ jointly convex just on sets of the type $K'\times K'\times \mathbb R$, insofar as $K'$ is such that $u_h(x) \in K'$ for a.e. $x$ and all $h$. 
\end{Remark}

In the sequel, taking also into account the models proposed in \cite[Example 5.23]{AFP}, we first state the properties inherited by the function $\Theta$ in Theorem \ref{mainthm} and then we provide some examples, essentially in the case where $f$ is jointly convex.

We observe that Definition \ref{Theta} easily entails that $\Theta$ has the following properties, see also \cite[Proposition 4.2]{GZ2}:
\begin{itemize}
\item[(i)]
$\Theta(O i, O j,O p)= \Theta(i,j, p)$

\noindent
for every  orthogonal matrix $O \in \mathbb R^{N \times N} $, $i,j \in \mathbb R^N$ and $p \in S^{N-1}$.
\item[(ii)] $\Theta(i,j,\cdot)$ is an even function, if $f$ is even in the last variable.
Moreover $\Theta(i,j,p)= \Theta(j,i, -p).$
\item[(iii)] If $f$ is continuous, then $\Theta$ is continuous on $\mathbb R^N \times \mathbb R
^N \times S^{N-1}$.
\item[(iv)] $\Theta(i,j, p)$ is subbaditive, in the sense that
$\Theta(i,j,p)\leq \Theta(i,k,p)+ \Theta(k,j,p)$, for every $p \in \mathbb S^{N-1}$.
\item[(v)] If $f$ is bounded, $\Theta$ is also bounded on $\mathbb R^N \times \mathbb R^N \times S^{N-1}$.
\end{itemize}

\begin{Examples}\label{EXAIMETA}

Let $a_1, a_2, b \in \mathbb R$, possible choices of a jointly convex function $f$ in Definition \ref{Theta} are the following, see \cite[Example 5.23]{AFP}: 
\begin{itemize}
\item[1] $f(a_1,a_2,b)= \left\{
\begin{array}{ll}
(g(a_1)+ g(a_2)) |b| & \hbox{ if } a_1 \not= a_2,\\
0 &\hbox{ if }a_1=a_2
\end{array}\right.$ 
and $g:\mathbb R \to [0, +\infty[$ is continuous
\item[2] $f(a_1, a_2, b)= \psi(|a_1-a_2|)\theta(b)$ with $ \psi$ lower semicontinuous, increasing and subadditive, and  $\theta$ even, positively 1-homogeneous and convex. It is worthwhile also to mention that for this class of functions  joint convexity and $BV$-ellipticity are equivalent  as proven in \cite{A2} (see Proposition 5.1 and subsequent observations therein).

Observe that $\theta(\cdot)= |\cdot|$ recovers the result, obtained with different techniques in \cite[Theorem 1.2]{GZ2}. See also Proposition \ref{risultatoconjoint} herein.

\item[3] $f(a_1, a_2, b)= \delta(a_1,a_2)\theta(b)$, with $\delta:\mathbb R\times \mathbb R \to [0, +\infty[$, continuous, positive, symmetric and satisfying the triangle inequality, (for instance $\delta(a_1, a_2)= |g(a_1)-g(a_2)|$, with $g: \mathbb R \to [0, +\infty[$ continuous, and $\theta(\cdot)=|\cdot|$.
\end{itemize}
\end{Examples}

As an application of Theorem \ref{mainthm} and Proposition \ref{risultatoconjoint} some existence results may be proven.

We emphasize that they strongly rely on some recent lower semicontinuity results for bulk energies in $SBD$ due to Ebobisse \cite{E} and to Lu and Yang (see \cite[Theorem 2.8 and Theorem 4.1]{LY}). 

\begin{Theorem}\label{minpb1}
Let $s>1$ and let $V:\Omega \times M^{N\times N}_{\rm sym} \to [0, +\infty)$ be a Carath\'eodory function satisfying
\begin{itemize}
\item for a.e. $x \in \O$, for every $\xi \in M^{N \times N}_{\rm sym}$,
\begin{equation}\label{16LY}
\frac{1}{C}|\xi|^s \leq V(x,\xi)\leq \rho(x)+ C(1+ |\xi|^s)
\end{equation}
for some constant $C>0$ and a function $\rho \in L^1(\O)$,
\item for a.e. $x_0 \in \O, V(x_0,\cdot)$ is symmetric quasiconvex, i.e.,
$$
V(x_0,\xi)\leq \frac{1}{|A|}\int_A V(x_0, \xi+ {\cal E} \varphi(x))dx
$$ 
for every bounded open subset $A$ of $\mathbb R^N$, for every $\varphi \in W^{1,\infty}_0(A;\mathbb R^N)$ and $\xi \in M^{N  \times N}_{\rm sym}$.  
\end{itemize}
Let $h \in L^1(\Omega;\mathbb R^N)$ and let $\{H(x)\}_{x \in \O}$ be a uniformly bounded family of closed subsets of $\mathbb R^N$.  Let $\Theta: \mathbb R^N \times \mathbb R^N \times \mathbb S^{N-1} \to [0, +\infty[$, be a continuous function as in (\ref{Thetadef}). Then the constrained minimum problem 
\begin{equation}\label{minpb32BCDM}
\displaystyle{\min_{\begin{array}{lll}
u \in SBD(\O),
[u]\cdot \nu_u \geq 0 \hbox{ a.e. in }J_u,\\
u(x) \in H(x) \hbox{ a.e. in }\O   
\end{array}}
\left\{\int_\O V(x, {\cal E}u)dx + {\cal H}^{N-1}(J_u) +\int_{J_u} \Theta(u^+,u^-, \nu_u)d {\cal H}^{N-1}+ \int_{\O}h \cdot u dx\right\}}
\end{equation}
admits a solution.
\end{Theorem}
\begin{proof}[Proof]
The hypotheses on $\{H(x)\}_{x \in \O}$ guarantee that every minimizing sequence $\{u_h\}$ is bounded in $L^\infty$. Therefore, by using the Direct Methods of the Calculus of Variations and, by virtue of Theorem \ref{compactnessBCDM} above and Theorem \ref{mainthm}, \cite[Theorem 2.8]{LY} we get a solution.
 \end{proof}

 By the same token invoking \cite[Theorem 4.1]{LY} the following result can be proven
 \begin{Theorem}\label{minpb32BCDM2}
 Let $V:\Omega \times \mathbb R^N \times M^{N  \times N}_{\rm sym} \to [0, +\infty]$. Assume that for every $x \in \Omega$ and for every $u \in \mathbb R^N$ and for every $A \in M^{N \times N}_{\rm sym}$:
 \begin{itemize}
 \item $V(x,u,\cdot)$ is convex and lower semicontinuous on $M^{N \times N}_{sym}$;
 \item $V(\cdot,u,A)$ is measurable in $\O$;
 \item for a.e. $x \in \O$ and for all $u \in \mathbb R^N$ and $\eta>0$ there exists $\delta >0$ such that
 $$
 V(x,u,\xi)-V(x,v,\xi) \leq \eta (1+ V(x,v,\xi))
 $$
 for all $v \in \mathbb R^N$ with $|u-v|\leq \delta$ and for all $\xi\in M^{N \times N}_{\rm sym}$;
 \item there exist $\gamma>0, s>1 $ such that 
 $$
 V(x,u,\xi) \geq \gamma |\xi|^s, \hbox{ for every }x \in \O \hbox{ and for every } (u,\xi) \in \mathbb R^N \times M^{N \times N}_{\rm sym}.
 $$
 \end{itemize}
Let $h \in L^1(\Omega;\mathbb R^N)$ and let $\{H(x)\}_{x \in \O}$ be a uniformly bounded family of closed subsets of $\mathbb R^N$.  Let $\Theta: \mathbb R^N \times \mathbb R^N \times \mathbb S^{N-1} \to [0, +\infty[$, be a continuous function as in (\ref{Thetadef}). Then the constrained minimum problem 
\begin{equation}\label{minpb32BCDMbis}
\displaystyle{\min_{\begin{array}{lll}
u \in SBD(\O),
[u]\cdot \nu_u \geq 0 \hbox{ a.e. in }J_u,\\
u(x) \in H(x) \hbox{ a.e. in }\O   
\end{array}}
\left\{\int_\O V(x,u, {\cal E}u)dx + {\cal H}^{N-1}(J_u) +\int_{J_u} \Theta(u^+,u^-, \nu_u)d {\cal H}^{N-1}+ \int_{\O}h \cdot u dx\right\}}
\end{equation}
admits a solution. 
 \end{Theorem}
 
Other choices of the forces $h$, appearing in the minimum problems above, are also possible: we refer to \cite{LY}. Analogously the function $\Theta$ can be chosen as in (\ref{Thetadef2}) and it is enough to invoke Proposition \ref{risultatoconjoint}.

\end{document}